\documentclass[12pt]{amsart}

\newtheorem{theorem}{Theorem}
\newtheorem{lemma}{Lemma}

\newtheorem*{thma}{Theorem A}
\newtheorem*{thmb}{Theorem B}
\newtheorem*{thm2'}{Theorem 2$^\prime$}

\renewcommand{\S}{\mathcal{S}}

\newcommand{\vphi}{\varphi}

\newcommand{\Hess}{\operatorname{Hess}}
\newcommand{\grad}{\operatorname{grad}}
\title{Injectivity Criteria for Holomorphic Curves in $\mathbb{C}^n$}

\thanks{The authors were supported in part by FONDECYT Grant \# 1030589.}

\author{M. Chuaqui \and P. Duren \and B. Osgood}

\address{P. Universidad Cat\'olica de Chile}
\email{mchuaqui@mat.puc.cl}
\address{University of Michigan}
\email{duren@umich.edu}
\address{Stanford University}
\email{osgood@stanford.edu}



\date{}

\begin{document}

\maketitle

\bibliographystyle{amsplain}

\begin{abstract}
Combining the definition of Schwarzian derivative for conformal
mappings between Riemannian manifolds given by Osgood and Stowe
with that for parametrized curves in Euclidean space given by
Ahlfors, we establish injectivity criteria for holomorphic curves
$\phi:\mathbb{D}\rightarrow\mathbb{C}^n$. The result can be
considered a generalization of a classical condition for
univalence of Nehari.

\end{abstract}


\section{Introduction} \label{section:intro}

Let $f:\mathbb{D}\rightarrow\mathbb{C}$ be a locally injective
holomorphic mapping defined in the unit disk, and let
\[
\S f=(f''/f')'-(1/2)(f''/f')^2
\]
be its Schwarzian derivative. A classical univalence criterion of
Nehari \cite{nehari:nehari-old-p} relates the size of $|Sf|$ to
the univalence of $f$. Nehari stated the result in the form:
\begin{equation} \label{eq:p-criterion}
|\S f(z)| \leq 2p(|z|)
\end{equation}
implies that $f$ is injective in $\mathbb {D}$, if $p(x)$ is a
positive, even, and continuous function defined for $x\in(-1,1)$,
with the properties
\begin{enumerate}
\item[(a)] $(1-x^2)^2p(x)$ is non-increasing for $x\in[0,1)$;
\item[(b)] the differential equation $u''+pu=0$ has no
nontrivial solutions with more than one zero in $(-1,1)$.
\end{enumerate}

Condition \eqref{eq:p-criterion} includes the criteria $|\S
f(z)|\leq \pi^2/2$ and $|\S f(z)|\leq 2(1-|z|^2)^{-2}$ from
\cite{nehari:schlicht}, but also many others. A function $p$
satisfying the hypotheses above will be referred to as a
\emph{Nehari function}. Later, in Section \ref{section:conf-sch},
we will also introduce the notion of an \emph{extremal Nehari
function}.

Before its connection with univalence was understood, the attributes
of the
Schwarzian that made it interesting are that it vanishes identically
precisely for M\"obius transformations,
\[
\S f(z) = 0 \quad \text{if and only if} \quad f(z) = \frac{az+b}{cz+d}\,,
\quad ad-bc \ne 0\,,
\]
and that it is invariant under post-composition with a M\"obius
transformation,
\[
S(f \circ g) = Sg
\]
if $f$ is M\"obius.
More generally, one has the chain rule
\begin{equation}
\S(f\circ g) = ((\S f)\circ g)(g')^2 + \S g\,.
\label{eq:S-chain-rule}
\end{equation}

Consider now a locally injective
holomorphic curve $\phi:\mathbb{D}\rightarrow\mathbb{C}^n$, $n \ge 1$.
Write $\phi =(f_1,\dots,f_n)$, with each $f_k$ holomorphic in $\mathbb{D}$,
and define the smooth real-valued function $\sigma$ on $\mathbb{D}$ by
\[
\sigma = \frac{1}{2}\log(|f_1'|^2+\cdots |f_n'|^2)\,.
\]
We define the Schwarzian derivative of $\phi$ to be
\begin{equation} \label{eq:hol-schwarzian}
\S\phi = 2(\sigma_{zz}-\sigma_z^2)\,,
\end{equation}
where
\[
\sigma_z = \frac{1}{2}\left(\frac{\partial\sigma}{\partial x} - i
\frac{\partial\sigma}{\partial y}\right)\,.
\]
This reduces to the classical Schwarzian when $n=1$, so there is no
ambiguity in using the same name and symbol. Further background on
this definition is in Section \ref{section:conf-sch}; it derives from a
generalization of the Schwarzian to conformal mappings of Riemannian manifolds.

A straightforward calculation based only on the definition
\eqref{eq:hol-schwarzian} together with   $\S$ vanishing on M\"obius
transformations shows that
\begin{equation}
\S(\phi \circ T) = ((\S \phi)\circ T)(T')^2 \label{eq:S-circ-T}
\end{equation}
when $T$ is a M\"obius transformation of $\mathbb{D}$. We will
need this in an number of places.  We do not consider $M\circ
\phi$ when $M$ is a M\"obius transformation of $\mathbb{R}^{2n}$,
nor do we have $\S(M\circ \phi)=\S\phi$, since, in general,
$M\circ \phi$ is not holomorphic and so its Schwarzian is not
defined (at least not so simply). However, there is a substitute
for M\"obius invariance that we will also need. It involves a
version of the Schwarzian introduced by Ahlfors, discussed in
Section \ref{section:S1}. A very general version of
\eqref{eq:S-chain-rule} is in Section \ref{section:conf-sch}.

Let $\Sigma = \phi(\mathbb{D})$, which we can regard as a (real) 2-dimensional
surface in $\mathbb{R}^{2n}$. The Gaussian curvature, $K(\phi(z))$, of $\Sigma$
at $\phi(z)$ is given by
\begin{equation} \label{eq:gauss-curvature}
K(\phi(z)) = -e^{-2\sigma(z)}\Delta\sigma(z)\,,
\end{equation}
and so is nonpositive. We shall prove:

\begin{theorem} \label{theorem:hol-p-criterion}
Let $p$ be a Nehari function and let $\phi: \mathbb{D} \rightarrow \mathbb{C}^n$
be holomorphic with $\phi' \ne 0$. If
\begin{equation} \label{eq:hol-p-criterion}
|\S \phi(z)| + \frac{3}{4}|\phi'(z)|^2|K(\phi(z)| \le 2p(|z|)\,,\quad z
\in \mathbb{D}\,,
\end{equation}
then $\phi$ is injective.
\end{theorem}
\noindent We recover Nehari's theorem if $n=1$, since then
$\phi(\mathbb{D}) \subset \mathbb{C}$ and $K(\phi(z)) =0$.

We are also able to describe just how injectivity fails on $\partial\mathbb{D}$,
and it does so in a rather special way. To have such a statement make sense it
is first necessary to know that a mapping satisfying Theorem
\ref{theorem:hol-p-criterion} extends continuously to the boundary. In Section
\ref{section:boundary} we will give a precise analysis of the situation, but as
a preliminary result we now state:

\begin{theorem} \label{theorem:boundary-extension-1}
A holomorphic curve $\phi:\mathbb{D} \rightarrow \mathbb{C}^n$ satisfying Theorem
\ref{theorem:hol-p-criterion} has an extension to $\overline{\mathbb{D}}$ that
is uniformly continuous in the spherical metric.
\end{theorem}

Thus a function $\phi$ satisfying the condition
\eqref{eq:hol-p-criterion} maps the unit circle to a continuous
closed curve $\Gamma\subset{\mathbb{C}}^n\cup\{\infty\}$. We say
that $\phi$ is an \emph{extremal function} for the criterion if
$\Gamma$ is \emph{not} a {simple} closed curve. In this case there
is a pair of points $\zeta_1, \zeta_2 \in \partial\mathbb{D}$ with
$f(\zeta_1)=f(\zeta_2)=P$, and one says that $P$ is a \emph{cut
point} of $\Gamma$. We now have the following characteristic
property of extremal mappings.

\begin{theorem}
\label{theorem:extremal}
Under the hypotheses of Theorem 1, suppose the
closed curve $\Gamma = \phi(\partial\Bbb D)$ is not simple and let $P$ be a
cut point.
Then there exists a Euclidean circle or line $C$ such that $C\setminus\{P\}
\subset
\Sigma$. Furthermore, equality holds in \eqref{eq:hol-p-criterion} along
$\phi^{-1}
(C \setminus\{P\})$.
\end{theorem}

In addition to these results, in Section \ref{section:covering} we
will derive a covering theorem for holomorphic curves satisfying
\eqref{eq:hol-p-criterion} that generalizes some one-dimensional
results. In Section \ref{section:example} we will construct
examples showing that the criterion is sharp.

\medskip

When we began our work on generalizing Nehari's theorem it was in a rather
different context, namely a harmonic mapping $f$ of $\mathbb{D}$ and its
Weierstrass-Enneper lift $\widetilde{f}$ mapping $\mathbb{D}$ to a  minimal
surfaces in $\mathbb{R}^3$, see \cite{cdo:harmonic-lift}. The lift $\widetilde{f}$
is a conformal mapping of $\mathbb{D}$, say with conformal factor $e^\sigma$,
and again there is a generalization of the Schwarzian derivative, $\S f =
2(\sigma_{zz}-\sigma_z^2)$, just as in \eqref{eq:hol-schwarzian}. One then has:
If $p$ is a Nehari function and
\begin{equation} \label{eq:harmonic-p-criterion}
|\S f(z)| + e^{2\sigma(z)}|K(\widetilde{f}(z))| \le 2p(|z|)\,, \quad
z\in\mathbb{D}\,,
\end{equation}
then $\widetilde{f}$ is injective in $\mathbb{D}$. Here
$K(\widetilde{f}(z))$ is the Gaussian curvature of the minimal surface
$\widetilde{f}(\mathbb{D})$ at $\widetilde{f}(z)$. There are also results
analogous to Theorems 2 and 3.

It was very surprising, to us at least, that such similar
statements hold in these two different settings.   Comparing
\eqref{eq:hol-p-criterion} and \eqref{eq:harmonic-p-criterion} the
most visible difference is in the multiple of the curvature terms,
$3/4$ in the former and $1$ in the latter. This is a reflection of
the shared nature of the proofs, with small changes in the
constants in some of the preliminary results. Things did not have to be this way, one might have thought.
The cause of this
commonality comes from the differential geometry of the
holomorphic curve as a surface in $\mathbb{R}^{2n}$, to be
explained in Section \ref{section:S1}.

We have certainly borrowed from the exposition in
\cite{cdo:harmonic-lift}, but while the two papers run a parallel
course in many -- but not all -- ways, we have tried to make this
paper reasonably self-contained. There are a few instances where we
refer to \cite{cdo:harmonic-lift} for details that would have been reproduced
verbatim here.

\section{Ahlfors' Schwarzian and the Second Fundamental Form} \label{section:S1}

We begin the same way as in \cite{cdo:harmonic-lift}, with
Ahlfors' Schwarzian for curves in $\mathbb{R}^m$ and its
relationship to curvature, but here we find an important
difference with what was done in the case of harmonic maps. Here
to make use of the properties of Ahlfors' operator to study
injectivity we must relate the second fundamental form of the
holomorphic curve as a surface in $\mathbb{R}^{2n}$ to its
Gaussian curvature.

Ahlfors \cite{ahlfors:schwarzian-rn} defined a Schwarzian
derivative for mappings $\varphi : (a,b)\rightarrow {\Bbb R}^m$ of
class $C^3$ with $\varphi'(x)\neq0$ by generalizing separately the
real and imaginary parts of the Schwarzian for analytic functions.
We only need the operator corresponding to the real part, which is
\[
S_1\varphi = \frac{\langle
\varphi',\varphi'''\rangle}{|\varphi'|^2} - 3\frac{\langle
\varphi',\varphi''\rangle^2}{|\varphi'|^4} +
\frac32\frac{|\varphi''|^2}{|\varphi'|^2}\,,
\]
where $\langle\cdot\, ,\cdot\, \rangle$ denotes the Euclidean
inner product. If $T$ is a M\"obius transformation of $\mathbb{R}^m$ then,
as Ahlfors showed,
\begin{equation}
S_1(T \circ \varphi)(t) = S_1\vphi(t)\,. \label{eq:S1-invariance}
\end{equation}
We also record the fact that if  $\gamma(t)$ is a smooth
function with ${\gamma}'(t) \ne 0$  then
\begin{equation*}
S_1(\vphi \circ \gamma)
 = ((S_1\vphi) \circ \gamma)(\gamma')^2 + \S
\gamma\,,\label{eq:S_1-chain-rule}
\end{equation*}
analogous to the chain rule \eqref{eq:S-chain-rule} for the analytic
Schwarzian.

As in the introduction, let $\phi :\mathbb{D} \rightarrow
\mathbb{C}^n$ be a holomorphic curve, and let
$\Sigma=\phi(\mathbb{D})$.  Ahlfors' Schwarzian enters the proof
of Theorem \ref{theorem:hol-p-criterion} for two reasons. First,
it is related to $\S \phi$ via the geometry of $\Sigma$ as a
surface in $\mathbb{R}^{2n}$. Second, bounds on $S_1\vphi$ imply
injectivity along curves, and this will be sufficient to prove
injectivity in $\mathbb{D}$. We take up the first point in this
section and the second point in Section \ref{section:injectivity}.

We recall the notion of the second fundamental form of a
submanifold. Let $M$ be a submanifold of $\mathbb{R}^m$ with the
metric induced by the Euclidean metric on $\mathbb{R}^m$. Let $D$
be the covariant derivative on $\mathbb{R}^m$ and let $D'$ be the
covariant derivative on $M$. If $X$ and $Y$ are vector fields
tangent to $M$ then $D_XY$ need not be tangent to $M$ but rather
has components tangent and normal to $M$:
\[
D_XY = D'_XY+I\!I(X,Y)\,.
\]
The normal component is the second fundamental form,
$I\!I(X,Y)$. It is a tensor.

For holomorphic curves the second fundamental form is related to
the Gaussian curvature of $\Sigma$ in the following way.

\begin{lemma} \label{lemma:II-gaussian}
Let
$\phi:\mathbb{D}\rightarrow\mathbb{C}^n$ be holomorphic with
$\phi'\neq 0$, and let $V(x)=\phi'(x)/|\phi'(x)|$, $x\in(-1,1)$. Then along $\phi$
the second fundamental form of $\Sigma=\phi(\mathbb{D})$ satisfies
\begin{equation} \label{eq:II-gaussian}
|I\!I(V,V)|^2=\frac12|K(\phi)|\,.
\end{equation}
\end{lemma}

\begin{proof}
We need to find the components of $D_VV$ tangent and normal to
$\Sigma$. First, using $V(x) =
\phi'(x)/|\phi'(x)|=e^{-\sigma(x)}\phi'(x)$ we have, along
$(-1,1)$,
\begin{equation} \label{eq:DVV}
D_VV=e^{-\sigma}(e^{-\sigma}\phi')'=
e^{-\sigma}(e^{-\sigma}\phi''-e^{-\sigma}\sigma_x\phi')
=e^{-2\sigma}(\phi''-\sigma_x\phi') = e^{-2\sigma}(\phi_{xx}-\sigma_x\phi_x)\, .
\end{equation}
Next let $Y$ be the vector field $e^{-\sigma}\phi_y$ along
$(-1,1)$. Since $\phi$ is conformal the pair $\{V,Y\}$ is
an orthonormal frame for $\Sigma$ along the curve $\phi((-1,1))$.
In order to determine the component of $D_VV$ normal to $\Sigma$
we will compute $\langle D_VV,V\rangle$ and $\langle
D_VV,Y\rangle$.

>From $\langle V,V\rangle= 1$ it follows that $\langle
D_VV,V\rangle=0$. Then,
\[
\langle D_VV,Y\rangle = e^{-3\sigma}\langle \phi_{xx}-\sigma_x\phi_x, \phi_y\rangle
=e^{-3\sigma}\langle \phi_{xx},\phi_y\rangle \, ,
\]
because $\langle \phi_x,\phi_y\rangle =0$. But $\phi$ is also
harmonic, hence
\[
\begin{aligned}
\langle D_VV,Y\rangle &= - e^{-3\sigma}\langle \phi_{yy},\phi_y\rangle
=-\frac12\, e^{-3\sigma}\frac{\partial}{\partial y}\langle \phi_y,\phi_y\rangle\\
&=-\frac12\,e^{-3\sigma}\frac{\partial}{\partial
y}(e^{2\sigma})-e^{-\sigma}\sigma_y \, .
\end{aligned}
\]
It follows that
\[
D_VV = -e^{-\sigma}\sigma_yY+I\!I(V,V) \, ,
\]
that is, from \eqref{eq:DVV},
\[
I\!I(V,V) =
e^{-2\sigma}(\phi_{xx}-\sigma_x\phi_x+\sigma_y\phi_y) \, .
\]
Therefore
\[
\begin{aligned}
e^{4\sigma}|\Pi(V,V)|^2
&=|\phi_{xx}|^2+e^{2\sigma}\sigma_x^2+e^{2\sigma}\sigma_y^2
-2\sigma_x\langle\phi_{xx},\phi_x\rangle+2\sigma_y\langle\phi_{xx},\phi_y\rangle\\
&=|\phi_{xx}|^2+e^{2\sigma}(\sigma_x^2+\sigma_y^2)-2\sigma_x\langle
\phi_{xx},\phi_x\rangle-2\sigma_y\langle\phi_{yy},\phi_y\rangle\\
&=
|\phi_{xx}|^2+e^{2\sigma}(\sigma_x^2+\sigma_y^2)-\sigma_x
\frac{\partial}{\partial x}(e^{2\sigma})
-\sigma_y\frac{\partial}{\partial y}(e^{2\sigma})\\
&= |\phi_{xx}|^2 - e^{2\sigma}(\sigma_x^2+\sigma_y^2)\,.
\end{aligned}
\]
More compactly,
\begin{equation} \label{eq:|II|^2}
e^{4\sigma}|I\!I(V,V)|^2=|\phi_{xx}|^2-e^{2\sigma}|\nabla\sigma|^2 \, .
\end{equation}

On the other hand, using $\phi=(f_1,\ldots,f_n)$ and  $\sigma =
(1/2\log\left(\,|f_1'|^2+\cdots+|f_n'|^2\right)$, together with
$\phi_x=\phi'=(f_1',\ldots,f_n')$ and
$\phi_{xx}=(f_1'',\ldots,f_n'')$, we obtain
\begin{equation} \label{eq:sigma_z}
2\sigma_{z}=\frac{\overline{f_1'}f_1''+\cdots\overline{f_n'}f_n''}
{|f_1'|^2+\cdots+|f_n'|^2} \, ,
\end{equation}
and after some algebra,
\[
|\nabla\sigma|^2=|2\sigma_{z}|^2=e^{-4\sigma}\left|\sum_i
\overline{f_i'}f_i''\right|^2
\, .
\]
This inserted in \eqref{eq:|II|^2} gives
\begin{equation} \label{eq:|II|^2-alternate}
e^{6\sigma}|\Pi(V,V)|^2=\sum_i\left|f_i''\right|^2\sum_j\left|f_j'\right|^2
-\sum_i\overline{f_i'}f_i''\sum_jf_j'\overline{f_j''}
=\sum_{i<j}\left|f_i'f_j''-f_j'f_i''\right|^2 \, . 
\end{equation}

Finally we compute $K(\phi(x))=-e^{-2\sigma(x)}\Delta\sigma(x)$. From \eqref{eq:sigma_z} it
follows that
\[
\frac12\Delta\sigma=2\sigma_{z\bar{{z}}}=\frac{|f_1''|^2+\cdots+|f_n''|^2}
{|f_1'|^2+\cdots+|f_n'|^2}-
\frac{\left|\overline{f_1'}f_1''+\cdots\overline{f_n'}f_n''\right|^2}
{\left(\,|f_1'|^2+\cdots+|f_n'|^2\right)^2}\,,
\]
and after some manipulation one obtains
\begin{equation} \label{eq:Laplacian-sigma}
\Delta\sigma=2e^{-4\sigma}\sum_{i<j}\left|f_i'f_j''-f_j'f_i''\right|^2
\, . 
\end{equation}
Comparing \eqref{eq:|II|^2-alternate} with
\eqref{eq:Laplacian-sigma}
 we deduce that
\[
|\Pi(V,V)|^2=\frac12|K(\phi)| \, ,
\]
as desired.

\end{proof}

Ahlfors' Schwarzian $S_1$ and the holomorphic Schwarzian $\S$ from
\eqref{eq:hol-schwarzian} appear together in the following relationship.

\begin{lemma} \label{lemma:ahlfors-and-hol-schwarzian}
Let $\phi :\mathbb{D} \rightarrow \mathbb{C}^n$ be holomorphic
with $\phi' \ne 0$. Let $\gamma(t)$ be a Euclidean arc-length
parametrized curve in $\mathbb{D}$ with curvature $\kappa(t)$, and
let $\vphi(t) = \phi(\gamma(t))$ be the corresponding
parametrization of $\Gamma = \phi(\gamma)$ on
$\Sigma=\phi(\mathbb{D})$. Let $V(t)$ be the Euclidean unit
tangent vector field along $\vphi(t)$, given by
\[
V(t) = \frac{\vphi'(t)}{|\vphi'(t)|} = \frac{\phi'({\gamma}(t)){\gamma}'(t)}
{|\phi'({\gamma}(t))|}\,.
\]
Then
\begin{equation} \label{eq:S1-curvature-general}
S_1\vphi(t)  = \text{Re}\{\S \phi({\gamma}(t))({\gamma}'(t))^2\}
+\frac{3}{4}|\phi'({\gamma}(t))|^2 K(\vphi(t)) +
\frac{1}{2}\kappa^2(t)\,.
\end{equation}
\end{lemma}

This is the most general form of the relationship between $S_1$ and $\S$.
Compare this formula to the one in Lemma 1 in \cite{cdo:harmonic-lift}.
We will also need \eqref{eq:S1-curvature-general} in the special case when
${\gamma}(t)$ is a diameter of $\mathbb{D}$, say from $-1$ to $1$.
In this case $\kappa = 0$ and we can write the equation as
\begin{equation} \label{eq:S1-curvature-special}
S_1 \phi(x) = \text{Re} \{\S \phi(x)\} + \frac{3}{4}|\phi'(x)|^2|K(\phi(x))|\,.
\end{equation}
Here, by $S_1 \phi(x)$ we mean $S_1$ applied to $\phi$ restricted to the
interval $(-1,1)$.

\begin{proof}[Proof of Lemma \ref{lemma:ahlfors-and-hol-schwarzian}]
Let $v(t) = |\vphi'(t)|$. The proof is based on a formula of Chuaqui and
Gevirtz in \cite{chuaqui-gevirtz:S1}, according to which
\begin{equation}
S_1\vphi = \left(\frac{v'}{v}\right)'-\frac{1}{2}\left(\frac{v'}{v}\right)^2
 +\frac{1}{2}v^2 k^2\,, \label{eq:S1-Ch-G}
\end{equation}
where $k$ is the Euclidean curvature of the curve $\vphi(t)$ in $\mathbb{R}^{2n}$.
We compute the terms on the right-hand side.

First, by definition,
\[
v(t) = |\vphi'(t)|= |\phi'({\gamma}(t))| = e^{\sigma({\gamma}(t))}\,,
\]
and hence
\[
\frac{v'}{v} = \langle \nabla\sigma\,,\,{\gamma}'\rangle\,,\quad
\left(\frac{v'}{v}\right)' = \Hess \sigma({\gamma}',{\gamma}') +
\langle\nabla\sigma\,,\,{\gamma}''\rangle\,,
\]
where
\[
\Hess\sigma =
\begin{pmatrix}
\sigma_{xx} & \sigma_{xy}\\
\sigma_{xy} & \sigma_{yy}
\end{pmatrix}
\]
is the Hessian matrix regarded as a bilinear form and ${\gamma}'$ is
identified with unit tangent vector $(x'(t),y'(t))$. Moreover,
with a similar identification, ${\gamma}'' = \kappa \mathbf{n}$,
where $\mathbf{n}$ is the unit normal to ${\gamma}(t)$. Thus
\begin{equation}
\left(\frac{v'}{v}\right)'-\frac{1}{2}\left(\frac{v'}{v}\right)^2
= \Hess \sigma({\gamma}',{\gamma}') + \kappa\langle\nabla\sigma\,,\,
\mathbf{n}\rangle -\frac{1}{2}\langle\nabla\sigma\,,\,{\gamma}'\rangle^2\,.
\label{eq:Ch-G-first-term}
\end{equation}

Next we work with the curvature term $k^2v$. We can write
\begin{equation}
k^2 = k_i^2 + |I\!I(V,V)|^2\,, \label{eq:k-and-ki}
\end{equation}
where $k_i$ is the intrinsic (geodesic) curvature of
$\vphi({\gamma}(t))$ on the surface $\vphi(\mathbb{D}) = \Sigma
\subset \mathbb{R}^{2n}$. Furthermore,
$\vphi\colon(\mathbb{D},e^{2\sigma}\mathbf{g}_0) \rightarrow
(\Sigma,\mathbf{g}_0)$ is a local isometry between $\Sigma$ with
the Euclidean metric, denoted here by $\mathbf{g}_0$, and
$\mathbb{D}$ with the conformal metric $e^{2\sigma}\mathbf{g}_0$,
and so $k_i = \hat{\kappa}$, the curvature of ${\gamma}(t)$ in the
metric $e^{2\sigma}\mathbf{g}_0$. In turn, a classical formula in
conformal geometry states that
\[
e^\sigma \hat{\kappa} = \kappa - \langle \nabla \sigma\,,\,
\mathbf{n}\rangle;
\]
see, for example, Section 3 in \cite{os:sch}. Combining these
formulas gives us
\begin{equation}
v^2k^2 = k^2e^{2\sigma} = \kappa^2 - 2\kappa\langle\nabla \sigma\,,\,
\mathbf{n}\rangle + \langle\nabla\sigma\,,\,\mathbf{n}\rangle^2 +
e^{2\sigma}|I\!I(V,V)|^2\,. \label{eq:Ch-G-second-term}
\end{equation}

For $S_1\vphi$ in \eqref{eq:S1-Ch-G} we combine \eqref{eq:Ch-G-first-term}
and \eqref{eq:Ch-G-second-term} and manipulate some terms to write
\[
\begin{aligned}
S_1\vphi &= \Hess(\sigma)({\gamma}',{\gamma}') -\frac{1}{2}\langle
\nabla\sigma\,,\,{\gamma}'\rangle^2+ \frac{1}{2}\langle\nabla\sigma\,,\,
\mathbf{n}\rangle^2 + \frac{1}{2}\kappa^2 + \frac{1}{2}e^{2\sigma}|I\!I(V,V)|^2\\
&= \Hess(\vphi)({\gamma}',{\gamma}')+ \frac{1}{2}|\nabla\sigma|^2 -
\langle\nabla\sigma\,,\,{\gamma}'\rangle^2 + \frac{1}{2}\kappa^2 +
\frac{1}{2}e^{2\sigma}|I\!I(V,V)|^2\\
&= \Hess(\sigma)({\gamma}',{\gamma}') -
\langle\nabla\sigma\,,\,{\gamma}'\rangle^2 -
\frac{1}{2}(\Delta\sigma - |\nabla\sigma|^2)
+ \frac{1}{2}\Delta \sigma + \frac{1}{2}\kappa^2
+ \frac{1}{2}e^{2\sigma}|I\!I(V,V)|^2
\end{aligned}
\]
Next, one finds by straight calculation (see also Section \ref{section:conf-sch}) that
\[
\Hess(\sigma)({\gamma}',{\gamma}') -\langle\nabla\sigma\,,\,
{\gamma}'\rangle^2 -\frac{1}{2}(\Delta\sigma - |\nabla\sigma|^2) =
\text{Re}\{\S \phi({\gamma})({\gamma}')^2\}
\]
while from \eqref{eq:gauss-curvature},
\[
\frac{1}{2}\Delta \sigma = -\frac{1}{2}e^{2\sigma}K(\phi) =
\frac{1}{2}e^{2\sigma}|K(\phi)|\,.
\]
Substituting these and $e^{2\sigma} = |\phi'|^2$ gives
\[
S_1\vphi(t)  = \text{Re}\{\S \phi({\gamma}(t))({\gamma}'(t))^2\} +
\frac{1}{2}|\phi'({\gamma}(t))|^2(K(\vphi(t)) +
|I\!I(V(t),V(t))|^2) + \frac{1}{2}\kappa^2(t)\,.
\]
Now appealing to Lemma \ref{lemma:II-gaussian} brings this into final form,
\[
S_1\vphi(t)  = \text{Re}\{\S \phi({\gamma}(t))({\gamma}'(t))^2\} +
\frac{3}{4}|\phi'({\gamma}(t))|^2 K(\vphi(t)) + \frac{1}{2}\kappa^2(t)\,.
\]
\end{proof}


\section{Injectivity, Extremal Functions,  and the Proofs of Theorems
\ref{theorem:hol-p-criterion} and
\ref{theorem:extremal}}\label{section:injectivity}

We recall the hypothesis of Theorem \ref{theorem:hol-p-criterion},
that $\phi$ satisfies \eqref{eq:hol-p-criterion},
\begin{equation*}
|\S \phi(z)| + \frac{3}{4}|\phi'(z)|^2|K(\phi(z))| \le 2p(|z|)\,,
\quad z \in \mathbb{D}\,.
\label{eq:hol-p-criterion-2}
\end{equation*}
The proof of injectivity rests on a result of Chuaqui and Gevirtz
\cite{chuaqui-gevirtz:S1} giving a criterion for univalence on (real)
curves in terms of $S_1$.

\begin{thma}
Let $p(x)$ be a continuous
function such that the differential equation $u''(x)+p(x)u(x)=0$
admits no nontrivial solution $u(x)$ with more than one zero in
$(-1,1)$. Let $\varphi : (-1,1)\rightarrow {\Bbb R}^m$ be a curve
of class $C^3$ with tangent vector $\varphi'(x)\neq 0$.  If
$S_1\varphi(x)\leq2p(x)$, then $\varphi$ is univalent.
\end{thma}

We pass immediately to
\begin{proof}[Proof of Theorem \ref{theorem:hol-p-criterion}]
If $\phi$ satisfies \eqref{eq:hol-p-criterion} then, from
\eqref{eq:S1-curvature-special}, along the diameter $(-1,1)$,
\[
S_1\phi(x) \le 2p(x)\,,
\]
and so $\phi$ is injective there by Theorem A. The same holds for
any rotation $\phi(e^{i\theta} z)$ of $\phi$ and hence $\phi$ is
injective along any diameter of $\mathbb{D}$.

Suppose now that $z_1$ and $z_2$ are distinct points not on a diameter.
Let $\gamma$ be the hyperbolic geodesic through $z_1$ and $z_2$.
By a rotation of $\mathbb{D}$ we may assume that $\gamma$ meets the
imaginary axis orthogonally at a point $i\rho$. The M\"obius transformation
\[
T(z) = \frac{z-i\rho}{1+i\rho z}
\]
maps $\mathbb{D}$ onto itself, preserves the imaginary axis,
and carries $\gamma$ to the diameter $(-1,1)$. The function
\[
\psi(z) = \phi(T(z))
\]
is a holomorphic reparametrization of $\Sigma = \phi(\mathbb{D})$
with $\psi'\ne 0$ and we claim that
\[
|\S \psi(x)| + \frac{3}{4}|\psi'(x)|^2|K(\psi(x))| \le 2p(x)\,,\quad -1 < x <1\,.
\]
If so, then $S_1\psi(x)  \le 2p(x)$ as above, whence $\psi$ is
injective along $(-1,1)$ and $\phi(z_1) \ne \phi(z_2)$.

For this, first note that
\[
|\psi'(x)| =|\phi'(T(x))|\,|T'(x)|
\]
while also
\[
\S \psi(x) = \S\phi(T(x))T'(x)^2\,,
\]
from \eqref{eq:S-circ-T}. Next, by hypothesis,
\[
|\S\phi(T(x))|\,|T'(x)|^2 + \frac{3}{4}|\phi'(T(x))|^2|T'(x)|^2
|K(\phi(T(x)))|\le 2p(|T(x)|)|T'(x)|^2\,,
\]
and so the claim will be established if we show
\begin{equation}
p(|T(x)|T'(x)|^2) \le p(|x|)\,,\quad -1 <x <1\,. \label{eq:injectivity-punchline}
\end{equation}
But now a simple calculation gives that 
\begin{equation}
|x| \le |T(x)| \label{eq:nehari-trick}
\end{equation}
 for this
particular M\"obius transformation, whence
\[
(1-|T(x)|^2)^2p(|T(x))| \le (1-x^2)^2p(|x|)
\]
by the assumption that $(1-x^2)^2p(x)$ is non-increasing for $x\in[0,1)$.
Furthermore
\[
\frac{|T'(x)|}{1-|T(x)|^2} = \frac{1}{1-x^2}
\]
for any M\"obius transformation of $\mathbb{D}$ onto itself. Thus
\[
|T'(x)|^2p(|T(x)|) = \frac{(1-|T(x)|^2)^2}{(1-x^2)^2}p(|T(x)|) \le p(|x|)\,,
\]
finishing the proof.
\end{proof}

The trick of passing from injectivity along diameters to the
general case via this special M\"obius transformation,
using \eqref{eq:nehari-trick} to obtain \eqref{eq:injectivity-punchline}, goes
back to Nehari. We used the same argument in
\cite{cdo:harmonic-lift}.

\medskip

Theorem \ref{theorem:extremal} states that if $\phi$
satisfies the univalence criterion \eqref{eq:hol-p-criterion}
and fails to be injective on $\partial \mathbb{D}$ (assuming continuous extension)
then it must do so in a particular way, that the surface
$\Sigma = \phi(\mathbb{D})$ contains a Euclidean circle
minus the cut point where injectivity fails. Moreover,
equality holds in \eqref{eq:hol-p-criterion} along the
preimage of the circle. These properties of such extremal
functions depend upon another result of Chuaqui and Gevirtz
in the same paper \cite{chuaqui-gevirtz:S1}. To state it we need
one additional construction, which will be used again in later sections.

If the function $p(x)$ of Theorem A is even, as will be the case
for a Nehari function, then the solution $u_0$
of the differential equation $u''+pu=0$ with initial conditions $u_0(0)=1$
and $u_0'(0)=0$ is also even, and therefore $u_0(x)\neq0$ on $(-1,1)$,
since otherwise it would have at least two zeros.  Thus the function
\begin{equation}
\Phi(x) = \int_0^x u_0(t)^{-2}\,dt\,, \qquad -1 < x < 1\,,
\label{eq:Phi}
\end{equation}
is well defined and has the properties $\Phi(0)=0$, $\Phi'(0)=1$,
$\Phi''(0)=0$, $\Phi(-x)=-\Phi(x)$, and  $\mathcal{S}\Phi=2p$.
Furthermore, $S_1\Phi=\mathcal{S}\Phi$ since
$\Phi$ is real-valued, and so $S_1\Phi=2p$ as well.

In terms of $\Phi$, the second Chuaqui-Gevirtz theorem is as follows.

\begin{thmb} Let $p(x)$ be an even function with the properties
assumed in Theorem A, and let $\Phi$ be defined as above.  Let
$\varphi: (-1,1)\rightarrow {\Bbb R}^m$ satisfy $S_1\varphi(x)\leq2p(x)$
and have the normalization $\varphi(0)=0$, $|\varphi'(0)|=1$, and
$\varphi''(0)=0$.  Then $|\varphi'(x)|\leq\Phi'(|x|)$ for $x\in(-1,1)$,
and $\varphi$ has an extension to the closed interval $[-1,1]$ that is
continuous with respect to the spherical metric.  Furthermore, there are
two possibilities:
\par $(i)$  If $\Phi(1)<\infty$, then $\varphi$ is univalent in $[-1,1]$
and $\varphi([-1,1])$ has finite length.
\par $(ii)$ If $\Phi(1)=\infty$, then either $\varphi$ is univalent in
$[-1,1]$ or $\varphi=R\circ\Phi$ for some rotation $R$ of ${\Bbb R}^m$.
\end{thmb}

We can now proceed with

\begin{proof}[Proof of Theorem \ref{theorem:extremal}]

Suppose that  $\phi$ satisfies \eqref{eq:hol-p-criterion} and fails
to be injective on
$\partial\mathbb{D}$, and let $\zeta_1,\zeta_2\in\partial\mathbb{D}$
be points for
which $\phi(\zeta_1)=\phi(\zeta_2)=P$. We first show that we can
form $\psi = \phi \circ T$ for a suitable  M\"obius transformation
of $\mathbb{D}$ onto itself, with $\psi$ still satisfying
\eqref{eq:hol-p-criterion} and with $\psi(1)= \psi(-1) = P$.
We refer to the calculations in the proof of Theorem
\ref{theorem:hol-p-criterion}, and we distinguish two cases.

Suppose first that $(1-x^2)^2p(x)$ is constant. Then
the condition \eqref{eq:hol-p-criterion} is fully invariant
under the M\"obius transformations of $\mathbb{D}$, and for a
M\"obius modification $\psi = \phi \circ T$ with
$T(1)=\zeta_1$, $T(-1)=\zeta_2$ we obtain  $\psi(1) = \psi(-1)=P$.

Now suppose that $(1-x^2)^2p(x)$ is not constant. Suppose also,
by way of contradiction, that $\zeta_1$ and $\zeta_2$ are not
on a diameter. Since \eqref{eq:hol-p-criterion} is, in any
case, invariant under rotations of $\mathbb{D}$, we may assume
that $\zeta_1$ and
$\zeta_2$ are both in the upper half plane and are symmetric in
the imaginary axis. Just as in the proof of Theorem \ref{theorem:hol-p-criterion},
let $T(z) = ({z-i\rho})/({1+i\rho z})$ with $T(-1)=\zeta_1$, $T(1)=\zeta_2$,
and let
$\psi=\phi\circ T$. Then \eqref{eq:hol-p-criterion} holds for $\psi$,
together with $\psi(1)=\psi(-1)=P$ and $S_1\psi(x)\leq 2p(x)$.
Moreover we must have a strict inequality $S_1\psi(x)<2p(x)$ on
some interval in $(-1,1)$ because $(1-x^2)^2p(x)$ is not constant.
However, this last statement stands in
contradiction with Theorem B. To be precise, there is a  M\"obius
transformation $M$ of $\Bbb{R}^{2n}$ such that $M\circ \psi$ satisfies
the hypotheses of Theorem B with $(M\circ \psi)(-1) = (M\circ \psi)(1)$,
and because $\psi$ is not injective on $[-1,1]$ we have, using \eqref{eq:S-circ-T}, 
$S_1(M\circ \psi) = S_1 \psi = 2p$. The contradiction shows that
$\zeta_1$ and $\zeta_2$ must lie on a diameter.

In all cases, by a suitable modification we can now assume
that the injectivity of $\phi$ on $\partial\mathbb{D}$ fails by
$\phi$ mapping the interval $[-1,1]$ to a closed curve on
$\phi(\overline{\mathbb{D}})$ with $\phi(-1)=\phi(1)=P$, and by a
post composition with a M\"obius transformation of $\mathbb{R}^{2n}$
we may further assume that $\phi$ is normalized as in Theorem B.
Then again by Theorem B, $\phi = V \circ \Phi$ for some M\"obius
transformation of $\mathbb{R}^{2n}$. Hence $S_1 \phi = S_1\Phi = 2p$
and $\phi$ maps $[-1,1]$ to a Euclidean circle or a line, since $\Phi$
maps $[-1,1]]$ onto $\mathbb{R}\cup\{\infty\}$ and $V$ preserves circles.
Finally, since $S_1\phi(x) = 2p(x)$ and $\phi$ satisfies
\eqref{eq:hol-p-criterion}, it follows from
\eqref{eq:S1-curvature-special} that
\[
|\S \phi(x)| + \frac{3}{4}|\phi'(x)|^2|K(\phi(x))| = 2p(x),\quad x \in [-1,1]\,.
\]
This concludes the proof.
\end{proof}


\section{A Covering Theorem} \label{section:covering}

We continue to assume that $\phi$ satisfies the injectivity criterion
\eqref{eq:hol-p-criterion}. In this section we will derive a lower
bound for the radius of a metric disk centered at $\phi(0)$ on $\Sigma$.
We will assume an additional normalization of $|\phi'(0)|=1$, and then
the lower bound will depend on $|\phi''(0)|$ and on a second extremal
function associated with the Nehari function $p$. We also need to assume
that $p$ is nondecreasing on $[0,1)$, which is the case for many examples.
The result we obtain is very much in line with those from classical
geometric function theory, see, for example, \cite{essen-keogh:schwarzian}.

Let $U$ be the solution of
\[
U''-pU=0\,,\qquad U(0)=1\,,\,\, U'(0)=0\,
\]
(note the minus sign in the differential equation) and define
\[
\Psi(x) = \int_0^x U(t)^{-2}\,dt\,.
\]

\begin{theorem} \label{theorem:covering} Let $\phi$ be a holomorphic curve
satisfying \eqref{eq:hol-p-criterion} with $|\phi'(0)|=1$, and
suppose that $p(x)$ is nondecreasing on $[0,1)$. Then
\begin{equation}
\min_{|z|=r}d_{\Sigma}(\phi(z), \phi(0)) \geq\,
\frac{2\Psi(r)}{2+|\phi''(0)|\Psi(r)} \,, \label{eq:min-d}
\end{equation}
where $d_{\Sigma}$ denotes distance on $\Sigma= \phi(\mathbb{D})$.  In particular,
$\Sigma$ contains a metric disk of radius
\[
\frac{2\Psi(1)}{2+|\phi''(0)|\Psi(1)}  
\]
centered at $\phi(0)$.
\end{theorem}

\begin{proof}  The proof relies ultimately on comparing solutions of
two differential equations, one involving the extremal $\Psi$, where
the relevant inequalities come to us by way of the formulas from Lemma
\ref{lemma:ahlfors-and-hol-schwarzian} and its proof. This requires some
preparation.

Let $z_r$ be a point on $|z|=r$ for which the minimum on the left
hand side of \eqref{eq:min-d} is attained. Since $\phi$ is
injective the minimum is positive and $z_r \ne 0$.  The geodesic,
$\Gamma$,  on $\Sigma$ that joins $\phi(0)$ with $\phi(z_r)$ is
contained in $\phi(\{|z|\leq r\})$ and we let
$\gamma=\phi^{-1}(\Gamma)$.  Write $\gamma(t)$ for the
parametrization of $\gamma$ by Euclidean arc-length, and let
$\varphi(t)=\phi(\gamma(t))$ be the corresponding parametrization
of $\Gamma$. Further, let
\[
v(t) = |\phi'(\gamma(t))|\,.
\]
Now  compare
the expressions for $S_1\varphi(t)$ in \eqref{eq:S1-curvature-general}
and \eqref{eq:S1-Ch-G}, using \eqref{eq:k-and-ki} relating the Euclidean
and intrinsic curvatures of $\Gamma$ and the second fundamental form of
$\Sigma$:
\[
\begin{split}
(v'/v)'-&\frac12(v'/)^2+\frac12\,v^2
\left(\, k_i^2+|I\!I(V,V)|^2\right) =\\
&{\rm Re}\{S\phi(\gamma)(\gamma')^2\}+\frac12\,v^2\left(\,
|K(\phi(\gamma))|+|I\!I(V,V)|^2\right)+\frac12\kappa^2 \, ,
\end{split}
\]
where $\kappa$ is the Euclidean curvature of $\gamma$ and $V$
is the Euclidean unit tangent vector field along $\Gamma$; all expressions are to be evaluated at $t$. Since $\Gamma$
is a geodesic $k_i=0$ and this becomes
\[
(v'/v)'-\frac{1}{2}(v'/v)^2 = {\rm
Re}\{\S\phi(\gamma)(\gamma')^2\}+
\frac{1}{2}v^2|K(\phi(\gamma))|+\frac{1}{2}\kappa^2\, .
\]
But the univalence criterion \eqref{eq:hol-p-criterion} implies
that
\[
{\rm Re}\{S\phi(\gamma)(\gamma')^2\} \geq -|S\phi(\gamma)| \geq
-2p(\gamma)+\frac34v^2|K(\phi(\gamma))| \, ,
\]
 hence
\begin{equation*}
\begin{aligned}
(v'/v)'-\frac12(v'/v)^2 &\geq
-2p(|\gamma|)+\frac54v(t)^2|K(\phi(\gamma))|+\frac12\kappa^2\\
&\geq
-2p(|\gamma|) \, . \label{eq:covering-comparison}
\end{aligned}
\end{equation*}
Now  let
\[
h(t)=\int_0^tv(\tau)\,d\tau \, .
\]
Then $\S h = (v'/v)'-\frac{1}{2}(v'/v)^2$, and from the preceding estimate,
\[
\S h(t) \geq -2p(|\gamma(t)|) \geq
-2p(t)\, , 
\]
the final inequality holding because $|\gamma(t)|\leq t$ and
we have assumed that $p(x)$ is nondecreasing on $[0,1)$.

The result $\S h(t) \ge -2p(t)$ is the main inequality we need
in order to apply the Sturm comparison theorem.
For this, first note that the
function $w=v^{-1/2}$ is the solution of
\[
w''+\displaystyle{\frac12}(Sh)w=0\,,\quad w(0)=1\,,\,w'(0) =
-\frac{1}{2} v'(0)\,.
\]
Next, consider also the solution $y(t)$ of
\[
y''-py=0\,,\quad y(0)=1\,,y'(0) = \frac{1}{2}|\phi''(0)| \, .
\]
Since
$
-\frac{1}{2}\S h(t) \le p(t)
$,
and also
\[
w'(0)=-\displaystyle{\frac12} v'(0) \leq
\displaystyle{\frac12}|\phi''(0)|\,,
\]
it follows by comparison that
\[
w(t) \le y(t)\,.
\]
Finally, observe that, explicitly,
$ y= (H')^{-1/2}$, where
\[
H=\frac{2\Psi}{2+|\phi''(0)|\Psi} \, .
\]
Consequently,
\[
h(t)=\int_0^tw^{-2}(\tau)\,d\tau \geq \int_0^ty^{-2}(\tau)\,d\tau =
\int_0^tH'(\tau)\,d\tau = H(t) \, ,
\]
and hence
\[
d_{\Sigma}(\phi(z_r),\phi(0))= \int_{\gamma}v \geq \int_0^r
v(t)\,dt = h(r) \geq H(r) \, .
\]
This completes the proof.

\end{proof}


\section{Conformal Schwarzians, Extremal Nehari Functions, and Convexity}
\label{section:conf-sch}

To advance further in the analysis of mappings satisfying
\eqref{eq:hol-p-criterion}, in particular to study continuous
extension to the boundary, we need estimates based on convexity.
This requires a notion of the Schwarzian for conformal metrics and an
associated differential equation. This section generally follows the
treatment of these ideas in \cite{cdo:harmonic-lift}, abbreviated somewhat
and modified to serve the case of holomorphic maps rather than lifts of
harmonic maps. We refer to that paper and also to \cite{os:sch} for
(many) more details.

Let $\mathbf{g}$ be a Riemannian metric on the disk $\Bbb D$.
We may assume that $\mathbf{g}$ is conformal to the Euclidean metric,
$\mathbf{g}_0=dx\otimes dx +dy \otimes dy= |dz|^2$. Let $\sigma$ be
a smooth function on $\Bbb D$ and form the symmetric 2-tensor
\begin{equation}
\Hess_{\mathbf{g}}(\sigma) - d\sigma \otimes d\sigma.
\label{eq:Hessian}
\end{equation}
Here $\Hess$ denotes the Hessian operator. If $\gamma(s)$ is an
arc-length parametrized geodesic for $\mathbf{g}$, then
\[
\Hess_{\mathbf{g}}(\sigma)(\gamma',\gamma') =
\frac{d^2}{ds^2}(\sigma\circ \gamma)\,.
\]
The Hessian depends on the metric, and since we will be
changing metrics we indicate this dependence by the subscript $\mathbf{g}$.

We now form
\[
B_\mathbf{g}(\sigma)={\Hess}_\mathbf{g}(\sigma)-d\sigma\otimes
d\sigma-\frac{1}{2}(\Delta_\mathbf{g} \sigma-||\grad_\mathbf{g}
\sigma||^2)\mathbf{g}\,.
\label{eq:B-sigma}
\]
The final term has been subtracted to make the trace zero.

This is the Schwarzian tensor of $\sigma$. Before explaining its
connection to conformal maps, metrics and the Schwarzian derivative,
first note that in standard Cartesian coordinates one can represent
$B_{\mathbf{g}}(\sigma)$ as a symmetric, traceless $2\times 2$ matrix,
say of the form
\[
\begin{pmatrix}
a & -b \\
-b & -a
\end{pmatrix}\,.
\]
Further identifying such a matrix with the complex number $a+bi$ then
allows us to associate the tensor $B_{\mathbf{g}}(\sigma)$ with $a+bi$,
and then
\[
||B_\mathbf{g_0}(\sigma)(z)||_\mathbf{g_0} = |a+bi|\,.
\]

A locally injective
holomorphic curve $\phi:\mathbb{D}\rightarrow\mathbb{C}^n$ is a
conformal mapping of $\mathbb{D}$ with the Euclidean metric into
$\mathbb{R}^{2n}$ with the Euclidean metric and if, as before,
we write  $\phi =(f_1,\dots,f_n)$ then the conformal factor is
\[
\phi^*(\mathbf{g_0}) = e^{2\sigma} \mathbf{g_0}\,,\quad \sigma =
\frac{1}{2}\log(|f_1'|^2+\cdots |f_n'|^2)\,.
\]
The Schwarzian derivative of $\phi$ is defined to be
\[
\mathcal{S}_\mathbf{g} \phi  = B_\mathbf{g}(\sigma)\,.
\]
When $n=1$ and $\sigma = \log|\phi'|$ we find that
\[
B_{\mathbf{g}_0}(\log|\phi'|)=\left( \begin{array}{rr} {\rm Re}\,
\mathcal{S}\phi & -{\rm Im}\,\mathcal{S}\phi \\ -{\rm Im}\,\mathcal{S}\phi
& -{\rm Re}\,\mathcal{S}\phi
\end{array}\right) \, ,
\]
writing the tensor in matrix form as above, where $\mathcal{S}\phi$ is
the classical Schwarzian derivative of $\phi$. When $n \ge 1$, identifying
the tensor with a complex number leads to
\[
\S\phi = 2(\sigma_{zz}-\sigma_z^2)\,,
\]
the definition that we gave in Section \ref{section:intro}.

The tensor $B_\mathbf{g} \sigma$ changes in a simple way if there is a
conformal change in the background metric $\mathbf{g}$. Specifically,
if $\widehat{\mathbf{g}} = e^{2\rho}\mathbf{g}$ then
\[
B_\mathbf{g}(\rho+\sigma) = B_\mathbf{g}(\rho)+B_{\widehat{\mathbf{g}}}(\sigma).
\]
This is actually a generalization of the chain rule \eqref{eq:S-chain-rule}
for the classical Schwarzian.  An equivalent formulation is
\begin{equation}
B_{\widehat{\mathbf{g}}}(\sigma -\rho) = B_\mathbf{g}(\sigma) -
B_\mathbf{g}(\rho)\,,
\label{eq:subtraction-formula}
\end{equation}
which is what we will need in later calculations.

Next, just as the linear differential equation $w''+(1/2)p w=0$ is
associated with $Sf=p$,  there is also a linear differential equation
associated with the Schwarzian tensor. If
\[
B_\mathbf{g}(\sigma) = p\,,
\]
where $p$ is a symmetric, traceless 2-tensor, then
$\eta=e^{-\sigma}$ satisfies
\begin{equation}
\Hess_\mathbf{g}(\eta) + \eta p = \frac{1}{2}(\Delta_\mathbf{g}\eta)\mathbf{g}\,.
\label{eq:Hess-eq}
\end{equation}

\medskip

We now turn to convexity. In this setting, a function $\eta$ is
convex relative to the metric $\mathbf{g}$ if
\[
\Hess_\mathbf{g} \eta \ge \alpha\mathbf{g}\,,
\]
where $\alpha$ is a nonnegative function.  This is equivalent to
\[
\frac{d^2}{ds^2}(\eta \circ \gamma) \ge \alpha \ge 0
\]
for any arc-length parametrized geodesic $\gamma$.

Convexity is an important notion for us because we will find that
an upper bound for
$\mathcal{S}_\mathbf{g}\phi$ coming from the injectivity criterion
\eqref{eq:hol-p-criterion} leads via \eqref{eq:subtraction-formula}
and \eqref{eq:Hess-eq} to just such a positive lower bound for the
Hessian of an associated function, and this is what we need to study
boundary behavior. This fact obtains, however, not relative to the
Euclidean metric but when the background metric $\mathbf{g}$ is a complete,
radial metric coming from an \emph{extremal} Nehari function.
We explain this now.

It follows from the Sturm comparison theorem that if $p$ is a Nehari
function then so is
a multiple $kp$ for any $k$ with $0<k<1$. This need not be so if
$k >1$ and we say that $p$ is an \emph{extremal Nehari function}
if $kp$ is \emph{not} a Nehari function for \emph{any} $k>1$.
For example, $p(x) = 1/(1-x^2)^2$ and $p(x)=\pi^2/4$ are both extremal
Nehari functions. In \cite{co:noncomplete} it was shown that some
constant multiple of each Nehari function is an extremal Nehari function.
Observe that since a holomorphic curve $\phi$ satisfying a condition
of the type $|\S \phi| +\cdots \le 2p$, as in \eqref{eq:hol-p-criterion},
also then satisfies $|\S \phi| +\cdots \le 2kp$ for any $k>1$ we may always
assume when \eqref{eq:hol-p-criterion} is in force that $p$ is an extremal
Nehari function.

There is another way to describe this situation in terms of the extremal
\emph{function} associated with a given $p$. Recall the definition from
\eqref{eq:Phi},
\[
\Phi(x) = \int_0^x u_0(t)^{-2}\,dt\,, \qquad -1 < x < 1\,,
\]
where $u_0$ is the solution of  $u''+pu=0$ with initial conditions $u_0(0)=1$
and $u_0'(0)=0$. We use $\Phi$ to form the radial conformal metric
\begin{equation}
\mathbf{g}_\Phi = \Phi'(|z|)^2|dz|^2 \label{eq:Phi-metric}
\end{equation}
 on $\mathbb{D}$.
It is implicit in  \cite{co:noncomplete}, without the terminology,
that the following conditions are equivalent:
\begin{enumerate}
\item[(a)] $p$ is an extremal Nehari function.
\item[(b)] $\Phi(1) = \infty$.
\item[(c)] The metric $\Phi'(|z|)^2|dz|^2$ is complete.
\end{enumerate}
We recall that for a complete metric any two points can be joined
by a geodesic and that a geodesic can be extended indefinitely. We
let $d_{\mathbf{g}_\Phi}$ be the distance in the ${\mathbf{g}_{\Phi}}$
metric and note that since ${\mathbf{g}_{\Phi}}$ is radial
\[
d_{\mathbf{g}_\Phi}(0,z) = \Phi(|z|)\,.
\]

The curvature of a radial metric of the form \eqref{eq:Phi-metric},
complete or not, can be expressed as
\begin{equation}
K_{\mathbf{g}_{\Phi}}(z) = -2 \Phi'(|z|)^{-2}(A(|z|) + p(|z|))\,,\quad r=|z|\,,
\label{eq:K-and-A}
\end{equation}
where
\begin{equation}
A(r) = \frac{1}{4}\left(\frac{\Phi''(r)}{\Phi'(r)}\right)^2+\frac{1}{2r}
\frac{\Phi''(r)}{\Phi'(r)}\,, \quad r\ge 0\,.
\label{eq:A}
\end{equation}
>From the properties of $\Phi$ it follows that $A(r)$ is continuous at
$0$ with $A(0)= p(0)$  and that the curvature is negative.
Thus, as we will need,
\begin{equation}
|K_{\mathbf{g}_{\Phi}}(z)| = 2 \Phi'(|z|)^{-2}(A(|z|) + p(|z|)) \,
.\label{eq:|K|}
\end{equation}
If the metric  \eqref{eq:Phi-metric} is complete, or equivalently comes
from an extremal Nehari function, then as was shown in \cite{co:noncomplete}
\begin{equation}
p(r) \le A(r)\,. \label{eq:p-and-A}
\end{equation}

All of these comments go into the proof of the following theorem.

\begin{theorem} \label{theorem:convexity}
Let $\phi$ satisfy the injectivity criterion \eqref{eq:hol-p-criterion}
for an extremal Nehari function $p$. Then
\begin{equation}
w(z)=\sqrt{\frac{\Phi'(|z|)}{|\phi'(z)|}} \label{eq:associated-fnc}
\end{equation}
satisfies
\begin{equation}
{\rm Hess}_{\mathbf{g}_{\Phi}}(w) \geq \,
\frac18\,w^{-3}|K_{\mathbf{g}_{\Phi}}|{\mathbf{g}_{\Phi}} \, , \label{eq:hess-w}
\end{equation}
In particular, $w$ is a convex function relative to the metric
${\mathbf{g}_{\Phi}}$.
\end{theorem}
Compare this result to the corresponding result, Theorem 4 in
\cite{cdo:harmonic-lift}, for lifts of harmonic mappings (in
\cite{cdo:harmonic-lift} we did not use the term ``extremal Nehari
function") where the constant in the inequality bounding the
Hessian from below is $1/4$ rather than $1/8$.

As in \cite{cdo:harmonic-lift} we formulate a separate lemma.

\begin{lemma} \label{lemma:convexity}
Let $\phi$ satisfy the injectivity criterion \eqref{eq:hol-p-criterion}
for an extremal Nehari function $p$ and let $\rho=\log|\Phi'|$. Then
\begin{equation}
\|B_{\mathbf{g}_{\Phi}}(\sigma - \rho)\|_{\mathbf{g}_{\Phi}}
+\frac{3}{4}e^{2(\sigma - \rho)}|K(\phi)| \le
\frac{1}{2}|K_{\mathbf{g}_{\Phi}}|\,. \label{eq:B(sigma-rho)}
\end{equation}
Here recall that $K(\phi(z))$ is the Gaussian curvature of
the surface $\Sigma = \phi(\mathbb{D})$ at $\phi(z)$.
\end{lemma}

\begin{proof} The proof is very much like the proof of Lemma 2
in \cite{cdo:harmonic-lift}, but to show how the assumptions
enter we will present the argument.

First note from \eqref{eq:|K|}  that, in terms of $\rho$, the
absolute value of the curvature of ${\mathbf{g}_{\Phi}}=
\phi'(|z|)^2|dz|^2=e^{2\rho(z)}|dz|^2$ is
\begin{equation}
|K_{\mathbf{g}_{\Phi}}(z)| = e^{-2\rho(z)}(A(|z|) + p(|z|)) \,
.\label{eq:|K|-2}
\end{equation}
Also, from \eqref{eq:subtraction-formula} we have
\[
B_{\mathbf{g}_{\Phi}}(\sigma-\rho) = B_{\mathbf{g}_0}(\sigma)-
B_{\mathbf{g}_0}(\rho) \, .
\]
Since ${\mathbf{g}_{\Phi}} = e^{2\rho}\mathbf{g}_0$, the norm
scales to give
\[
\|B_g(\sigma-\rho)\|_{\mathbf{g}_\Phi}=
e^{-2\rho}\|B_{\mathbf{g}_0}(\sigma)-B_{\mathbf{g}_0}
(\rho)\|_{\mathbf{g}_0}=e^{-2\rho}|B_{\mathbf{g}_0}
(\sigma)-B_{\mathbf{g}_0}(\rho)| \, ,
\]
where in the last equation we have identified the Euclidean norm of
the tensor with the magnitude of the corresponding complex number.

Next, a calculation (see also \cite{co:noncomplete}) produces
\[
B_{\mathbf{g}_0}(\sigma)-B_{\mathbf{g}_0}(\rho)=
\zeta^2S\phi(z)+A(|z|)-p(|z|)  \; ,
\quad \zeta=\frac{z}{|z|} \,.
\]
In light of these statements, establishing \eqref{eq:B(sigma-rho)}
is equivalent to
\[
\left|\zeta^2S\phi(z)+A(|z|)-p(|z|)\right|+\frac34\,e^{2\sigma(z)}
|K(\phi(z))| \leq A(|z|)+p(|z|) | \,.
\]
This in turn  follows from the assumption that $\phi$ satisfies the
injectivity criterion \eqref{eq:hol-p-criterion} and, crucially,
from the inequality \eqref{eq:p-and-A}:
\[
\begin{aligned}
\left|\zeta^2S\phi(z)+A(|z|)-p(|z|)\right|+\frac34\,e^{2\sigma}|K|
&\leq |\zeta^2S\phi(z)|+|A(|z|)-p(|z|)|+\frac34\,e^{2\sigma}|K|\\
&= |S\phi(z)|+\frac34\,e^{2\sigma}|K|+A(|z|)-p(|z|)\\
& \leq A(|z|)+p(|z|) | \, .
\end{aligned}
\]

\end{proof}

The deduction of Theorem \ref{theorem:convexity} from Lemma
\ref{lemma:convexity}, relies on \eqref{eq:Hess-eq}. Write
$w=e^{(\rho-\sigma)/2}$, $v = w^2 = e^{\rho-\sigma}$, and then,
according to \eqref{eq:Hess-eq},
\[
\Hess_{\mathbf{g}_{\Phi}} v +vB_{\mathbf{g}_{\Phi}}(\sigma-\rho)=
\frac{1}{2}(\Delta_{\mathbf{g}_{\Phi}} v)\mathbf{g}\,.
\]
With this, the proof is almost word-for-word the same as the
corresponding proof in \cite{cdo:harmonic-lift}, and we omit the details.
Instead, let us present one consequence of Theorem \ref{theorem:convexity}
here, with more to come in the next section.

\begin{lemma} \label{lemma:unique-critical-point}
Under the assumptions of Theorem \ref{theorem:convexity}, if $w$ has at
least two critical points then the range of $\phi$ lies in a plane.
\end{lemma}

\begin{proof}
Suppose $z_1$ and $z_2$ are critical points of $w$. Then, because
$w$ is convex, $w(z_1)$ and $w(z_2)$ are absolute minima, and so
is every point on the geodesic segment $\gamma$ (for the metric
$\mathbf{g}_\Phi$) joining $z_1$ and $z_2$ in $\mathbb{D}$. Hence
${\rm Hess}_g(w)(\gamma', \gamma')=0$, which from
\eqref{eq:hess-w} implies that $|K|\equiv 0$ along
$\Gamma=\phi(\gamma)$. From \eqref{eq:Laplacian-sigma} in Section
\ref{section:S1}  we have that
\[
|K| = 2e^{-6\sigma}\sum_{i<j}\left|f_i'f_j''-f_j'f_i''\right|^2\, ,
\]
hence $f_i'f_j''-f_j'f_i''=0$ along $\gamma$, for all $i<j$. By
analytic continuation, $f_i'f_j''-f_j'f_i''=0$ everywhere, and
using that $\phi'\neq 0$, it follows that for some $i$ and all $j$
there are constants $a_j, b_j$ such that $f_j=a_jf_i+b_j$. This
proves the lemma.

\end{proof}
A corresponding result for harmonic maps is Lemma 3 in \cite{cdo:harmonic-lift}.


\section{Boundary Behavior and the Proof of Theorem
\ref{theorem:boundary-extension-1}} \label{section:boundary}

We now study the boundary behavior for functions $\phi$ satisfying
\eqref{eq:hol-p-criterion}. As in the results just above we suppose
that $p$ is an extremal Nehari function and we set
\[
w(z)=\sqrt{\frac{\Phi'(|z|)}{|\phi'(z)|}}
\]
where $\Phi$ is the extremal function associated with $p$. We just
saw that if $w$ has at least two critical points then the range of
$\phi$ lies in a plane, and so the setting is effectively that of
an analytic function satisfying Nehari's criterion
\eqref{eq:p-criterion}. The boundary behavior in this case has
been thoroughly studied; see \cite{gehring:gehring-pommerenke},
\cite{co:gp} and also the summary of the classical results in
\cite{cdo:harmonic-lift}.

We next consider the situation when $w$ has a unique critical point,
and here the basic estimate is as follows.

\begin{lemma} \label{lemma:distortion}
If $w$ has a unique critical point then there are positive constants
$a$ and $b$ and a number $r_0$, $0<r_0<1$, such that
\begin{equation}
|\phi'(z)| \le \frac{\Phi'(|z|)}{(a\Phi(|z|)+b)^2}\,,\quad r_0<|z|<1\,.
\label{eq:distortion}
\end{equation}
\end{lemma}

\begin{proof}
Let $z_0$ be the unique critical point of
$w$. Let $\gamma(s)$ be an arc-length parametrized geodesic in the
metric $\mathbf{g}$ starting at $z_0$ in a given direction. Let
$\widetilde{w}(s)=w(\gamma(s))$. Now  the critical point is unique,
and therefore
$\widetilde{w}'(s)>0$ for all $s>0$. Thus there is an
$s_0>0$ and an $a>0$ such that $\widetilde{w}'(s) > a$ for all $s>s_0$.
This implies that $\widetilde{w}(s) > as+b$ for some positive constant $b$ and
$s>s_0$. It is easy to see from compactness that the constants
$s_0$, $a$, and $b$ in this estimate can be made uniform, independent of the
direction of the geodesic starting at $z_0$. In other words,
\[
w(z) \geq ad_{\mathbf{g}_\Phi}(z,z_0)+b
\]
for all $z$ with $d_{\mathbf{g}_\Phi}(z,z_0)>s_0$. By renaming the constant
$b$ and for suitable $r_0$ we will then have
\[
w(z) \geq ad_{\mathbf{g}_\Phi}(z,0)+b
\]
for all $z$ with $1>|z| > r_0$. The theorem follows from the
definition of $w$ because $d_{\mathbf{g}_\Phi}(z,0)=\Phi(|z|)$.
\end{proof}

The estimate in this lemma allows one to deduce that $\phi$ has a
continuous extension to $\overline{\mathbb{D}}$, and the argument
is just as in \cite{cdo:harmonic-lift}. We will give only a few details
here, enough for a more precise accounting of the regularity of the
extension (also as in \cite{cdo:harmonic-lift}).

Since the function $(1-x^2)^2p(x)$ is positive and
decreasing on $[0,1)$, we can form
\[
\lambda= \lim_{x\rightarrow 1}(1-x^2)^2p(x)\,.\] It was shown in
\cite{co:noncomplete} that $\lambda\leq 1$, and that $\lambda=1$
if and only if  $p(x)=(1-x^2)^{-2}$. In this case, the function
$\Phi$ is given by
\[
\Phi(z)=\frac12 \log\frac{1+z}{1-z}\,.
\]
Thus \eqref{eq:distortion} amounts to
\[
|\phi'(z)| \leq
\frac{1}{(1-|z|^2)\left(\displaystyle{\frac{a}{2}\log\frac{1+|z|}{1-|z|}}+b\right)^2}
\; ,  \quad |z|> r_0 \, .
\]
>From this, the technique of integrating along hyperbolic segments
in $\mathbb{D}$, see also \cite{gehring:gehring-pommerenke}, leads to
\begin{equation}
|\phi(z_1)-\phi(z_2)| \leq \, C\left(\log\frac{1}{|z_1-z_2|}\right)^{-1} \; ,
\label{eq:cont-ext-1}
\end{equation}
for some constant $C$ and points $z_1$, $z_2$ for which the
hyperbolic geodesic segment joining them is contained in the
annulus $r_0<|z|<1$. This implies that $\phi$ is uniformly
continuous in the closed disk, and its continuous extension also
satisfies \eqref{eq:cont-ext-1}. Thus when $\lambda =1$ the extension
has a logarithmic modulus of continuity.

Suppose now that $\lambda<1$. We appeal to a result from \cite{co:noncomplete},
according to which
\[
\lim_{x\rightarrow 1}(1-x^2)\frac{\Phi''}{\Phi'}(x) =
2(1+\sqrt{1-\lambda})=2\mu \, .
\]
Note that $1<\mu\leq 2$.  It follows that for any
$\epsilon>0$ there exists $0<x_0<1$ such that
\[
\frac{\mu-\epsilon}{1-x} \leq \frac{\Phi''}{\Phi'}(x) \leq
\frac{\mu+\epsilon}{1-x}\; ,  \quad x> x_0 \,,
\]
which implies that
\[
\frac{1}{(1-x)^{\mu-\epsilon}} \leq \Phi'(x) \leq
\frac{1}{(1-x)^{\mu+\epsilon}}\; ,  \quad x> x_0 \, ,
\]
Then
\[
\frac{\Phi'(x)}{(a\Phi(x)+b)^2} \leq \frac{C}{(1-x)^{\alpha+3\epsilon}} \, ,
\]
where $\alpha=2-\mu=1-\sqrt{1-\lambda}$ and $C$ depends on $a$, $b$
and the values of $\Phi$ at $x_0$. This estimate, together
with the technique of integration along hyperbolic
segments, implies that
\[
|\phi(z_1)-\phi(z_2)| \leq C|z_1-z_2|^{1-\alpha-3\epsilon}
= C|z_1-z_2|^{\sqrt{1-\lambda}-3\epsilon} \, ,
\]
for all points
$z_1$,  $z_2$ for which the hyperbolic geodesic segment joining them
is contained in the annulus $\max\{r_0,x_0\}<|z|<1$. This shows
that $\phi$ admits a continuous extension to the closed disk, with
at least a H\"older modulus of continuity.

We also point out that if one has the additional information that $x=1$ is a
regular singular point of the differential equation $u''+pu=0$,
then from an analysis of the Frobenius solutions at
$x=1$ one can deduce that
\[
\Phi'(x) \sim \frac{1}{(1-x)^{\mu}} \; , \: x\rightarrow 1 \, .
\]
This then provides exact H\"older continuous extension when
$\lambda >0$ and a Lipschitz continuous extension when $\lambda=0$.

\medskip

All of this discussion has been under the assumption that $w$
has a unique critical point. The argument in the case where $w$
has no critical points,
though still based on convexity, requires additional work. This,
too, is very close to what was done in \cite{cdo:harmonic-lift},
so we only sketch the key points.

As will be explained momentarily, it is necessary to consider shifts
$T\circ\phi$ of the holomorphic curve $\phi$ by M\"obius transformations
$T$ of  $\mathbb{R}^{2n}$.
The composition $T\circ\phi$
will not, in general, be holomorphic, though it is still conformal as
a mapping of $\mathbb{D}$ into $\mathbb{R}^{2n}$.  Write the corresponding
conformal metric on the disk as $e^{2\tau}\mathbf{g}_0$ and, restricting
$\tau$ to the radial segment $re^{i\theta}$, $0 \le r <1$, let
\[
\Upsilon_\theta(r) = e^{\tau(re^{i\theta})}
\]
Also, let $s=\Phi(r)$ be the arc-length parameter of $[0,1)$ in
the metric $\mathbf{g}_\Phi$, so that $r=\Phi^{-1}(s)$. Replacing
Theorem \ref{theorem:convexity}, along radial segments, we find
that the function
\[
\omega_\theta(s) = \left\{\frac{\Phi'(\Phi^{-1}(s))}
{\Upsilon_\theta(\Phi^{-1}(s))}\right\}^{1/2}
\]
is convex, meaning in this case simply that $\omega_\theta''(s) \ge 0$.

This is Lemma 5 in \cite{cdo:harmonic-lift}, in a slightly
different notation, and we will not give the (identical) proof.
The reason, however, why one can compose with a M\"obius
transformation of the range and still get a convexity result on
radial segments is that $S_1(T\circ \phi) = S_1\phi$ on radial
segments, from \eqref{eq:S-circ-T}. Then via Lemma
\ref{lemma:ahlfors-and-hol-schwarzian}, bounds on $\S\phi$ entail
bounds on $S_1(T\circ \phi)$ and the convexity of $\omega(s)$ can
be deduced from such bounds for $S_1$.

Now one shows, as in \cite{cdo:harmonic-lift}, that for any $\theta_0$ it
is possible to choose a M\"obius transformation $T$ so that
$\omega_{\theta_0}'(0) >0$. Therefore by convexity $\omega_{\theta_0}(s)
\ge as+b$, $a$, $b >0$, and then
\[
\Upsilon_{\theta_0}(r) \le \frac{\Phi'(r)}{(a\Phi(r)+b)^2}\,,
\]
which provides a substitute for \eqref{eq:distortion}. By continuity this estimate
holds for $\theta$ near $\theta_0$, so in a small angular sector about
the radius $re^{i\theta_0}$. This in turn implies that $T\circ \phi$ has a
continuous extension to the part of $\overline{\mathbb{D}}$ in the sector.
Since $\theta_0$ was arbitrary and since we allow for M\"obius transformations
of the range, we obtain an extension of $\phi$ to $\overline{\mathbb{D}}$ that
is continuous in the spherical metric.


\section{Examples} \label{section:example}

In this section we present some examples to show that the injectivity
criterion \eqref{eq:hol-p-criterion} is sharp. In the setting of lifts
of harmonic maps the corresponding  examples were provided by mappings
into a catenoid in $\mathbb{R}^3$. Surprisingly, the formulas are similar
here, though some of the analytical
details are different.

\bigskip

\noindent{\bf Example 1.} Let $p(x)=\pi^2/4$. Then the criterion
\eqref{eq:hol-p-criterion} becomes
\begin{equation}
|S\phi|+\frac34|\phi'|^2|K|\leq \frac{\pi^2}{2} \,
.\label{eq:example-1}
\end{equation}
Define $\phi:\mathbb{D}\rightarrow\mathbb{C}^2$ to be
\[
\phi(z)=( c\, e^{\pi z},\, e^{-\pi z} ) \, ,
\]
where the constant $c$ is to be chosen later. Then
\[
e^{2\sigma} = \pi^2( c^2\,e^{2\pi x} + e^{-2\pi x}) \, .
\]
Straightforward calculations produce
\[
\sigma_x = \pi\,\frac{c^2e^{4\pi x}-1}{c^2e^{4\pi x}+1} \, ,
\]
and
\[
\sigma_{xx}= \frac{8\pi^2c^2e^{4\pi x}}{(c^2e^{4\pi x}+1)^2} \, .
\]
Hence for the Schwarzian,
\[
\S\phi = 2(\sigma_{zz}-\sigma^2_z) =
\frac12(\sigma_{xx}-\sigma_x^2) = \frac{4\pi^2c^2e^{4\pi
x}}{(c^2e^{4\pi x}+1)^2}-\frac{\pi^2}{2} \left( \frac{c^2e^{4\pi
x}-1}{c^2e^{4\pi x}+1}\right)^2 \, .
\]

If $c$ is chosen large enough, then
\[
|\S\phi| = \frac{\pi^2}{2} \left( \frac{c^2e^{4\pi
x}-1}{c^2e^{4\pi x}+1}\right)^2-\frac{4\pi^2c^2e^{4\pi
x}}{(c^2e^{4\pi x}+1)^2} \, .
\]
Thus
\[
|S\phi|+\frac34|\phi'|^2|K| = |S\phi|+\frac34\,\sigma_{xx} =
\frac{\pi^2}{2} \left( \frac{c^2e^{4\pi x}-1}{c^2e^{4\pi
x}+1}\right)^2+\frac{2\pi^2c^2e^{4\pi x}}{(c^2e^{4\pi x}+1)^2} =
\frac{\pi^2}{2} \, .
\]
Therefore, equality holds in \eqref{eq:example-1} everywhere, and
$\phi$ is injective, but barely, since $\phi(1)=\phi(-1)$.

\bigskip

\noindent {\bf Example 2.} The previous example can be extrapolated to give a
general construction.  Let $p$
be a Nehari function with the additional property that it is the
restriction to $(-1,1)$ of an analytic function in the disk $p(z)$
that satisfies $|p(z)|\leq p(|z|)$. Typical examples are
$p(z)=(1-z^2)^{-2}$ and $p(z)=2(1-z^2)^{-1}$. The extremal map $\Phi$
is then analytic and univalent in the disk, and satisfies
$\S\Phi(z)=2p(z)$ there. Moreover, the image $\Phi(\mathbb{D})$ is a
parallel strip like domain, symmetric with respect to the real and imaginary
axes, and containing
the entire real line. Let
\[
f(z)=\frac{c\Phi(z)+i}{c\Phi(z)-i} \, ,
\]
where $c>0$ is to be chosen later and sufficiently small so that
$i/c \notin \Phi(\mathbb{D})$ (it can be shown that the map $\Phi$ is
always bounded along the imaginary axis, see \cite{co:gp}). The function $f$ maps
$\mathbb{D}$ onto a simply-connected domain containing the unit
circle minus the point 1. The smaller the value of $c$ the thinner the image of
$f$.

Define $\phi:\mathbb{D}\rightarrow\mathbb{C}^2$  by
\[
\phi(z)=(f(z),\frac{1}{f(z)}\,) \, .
\]
Then
\[
e^{2\sigma}=|f'(z)|^2\left(1+\frac{1}{|f(z)|^4}\right) \,,
\]
and a lengthy calculation results in
\begin{equation}
\S\phi=2(\sigma_{zz}-\sigma_z)=\S\Phi+6\frac{(\overline{f}f')^2}{(1+|f|^4)^2}
\, , \label{eq:Sf-example-2}
\end{equation}
and
\begin{equation}
e^{2\sigma}|K|=8\frac{|ff'|^2}{(1+|f|^4)^2} \, . \label{eq:sigma-example-2}
\end{equation}
Condition \eqref{eq:hol-p-criterion} then reads
\begin{equation}
\left|\S\Phi+6\frac{(\overline{f}f')^2}{(1+|f|^4)^2} \,\right|
+6\frac{|ff'|^2}{(1+|f|^4)^2} \leq \S\Phi(|z|) \, .\label{eq:criterion-example-2}
\end{equation}
Suppose, for example, we let $p(z)=(1-z^2)^{-2}$, for which the
extremal function is
\[
\Phi(z)=\frac12\log\frac{1+z}{1-z}\,.
\]
Then
$\Phi'(z)=(1-z^2)^{-1}$  and $f'=-2ic\Phi'/(c\Phi-i)^2$, and after some
simplifications \eqref{eq:criterion-example-2} becomes
\[
\left|\frac{2}{(1-z^2)^2}-\frac{24c^2(c\overline{\Phi}-i)^2}
{(1-z^2)^2(c\overline{\Phi}+i)^2(c\Phi-i)^4(1+|f|^4)}\,\right|
+\frac{24c^2|c\Phi+i|^2}{|1-z^2|^2|c\Phi-i|^6(1+|f|^4)} \leq\,
\frac{2}{(1-|z|^2)^2} \, ,
\]
 which further reduces to
\begin{equation}
\left|1-\frac{12c^2(1+c^2\Phi^2)^2}{(|c\Phi-i|^4+|c\Phi+i|^4)^2}\,\right|
+\frac{12c^2|1+c^2\Phi^2|^2}{(|c\Phi-i|^4+|c\Phi+i|^4)^2} \leq \,
\frac{|1-z^2|^2}{(1-|z|^2)^2} \, .\label{eq:reduced-criterion}
\end{equation}
 Let
\[
\zeta=\displaystyle{\frac{12c^2(1+c^2\Phi^2)^2}{(|c\Phi-i|^4+|c\Phi+i|^4)^2}}\,.
\]
In order to guarantee \eqref{eq:reduced-criterion}, we need the following
estimates for
$|1-{\rm Re}\{\zeta\}|$ and $|{\rm Im}\{\zeta\}|$.

\begin{lemma} \label{lemma:example-estimates}
If $c$ is small then there exist
absolute constants $A$, $B$, $C$ such that
\begin{equation}
|1-{\rm Re}\{\zeta\}| \leq 1-|\zeta| +Ac^4|{\rm Im}\{{\Phi}\}|^2 \, ,
\label{eq:hard-estimate-1}
\end{equation}
\begin{equation}
|{\rm Im}\{\zeta\}| \leq Bc^3|{\rm Im}\{{\Phi}\}| \, ,\label{eq:hard-estimate-2}
\end{equation}
and
\begin{equation}
|1-\zeta| \leq  1-|\zeta| +Cc^4|{\rm Im}\{{\Phi}\}|^2 \, .\label{eq:hard-estimate-3}
\end{equation}
\end{lemma}

\begin{proof}
It is clear that $|\zeta| < 1$ if $c$ is small, hence $|1-{\rm
Re}\{\zeta\}| = 1-{\rm Re}\{\zeta\}$. Thus
\eqref{eq:hard-estimate-1} amounts to
\begin{equation}
|\zeta|- {\rm Re}\{\zeta\} \leq Ac^4|{\rm
Im}\{{\Phi}\}|^2 \, . \label{eq:hard-estimate-1a}
\end{equation}
We have
\[
\begin{aligned}
|\zeta|-{\rm Re}\{\zeta\} &=
\frac{12c^2}{(|c{\Phi}-i|^4+|c{\Phi}+i|^4)^2}\left[\, |1+c^2{\Phi}^2|^2-{\rm
Re}\{(1+c^2{\Phi}^2)^2\}\right]\\
& = \frac{12c^6( |{\Phi}|^4-{\rm
Re}\{{\Phi}^4\})}{(|c{\Phi}-i|^4+|c{\Phi}+i|^4)^2}\\
& =
-\frac{6c^6({\Phi}^2-\overline{{\Phi}}^{\,2})^2}{(|c{\Phi}-i|^4+|c{\Phi}+i|^4)^2}\\
&= \frac{24c^6({\Phi}+\overline{{\Phi}})^2|{\rm
Im}\{{\Phi}\}|^2 }{(|c{\Phi}-i|^4+|c{\Phi}+i|^4)^2} \,
\end{aligned}
\]
which shows \eqref{eq:hard-estimate-1a} and thus
\eqref{eq:hard-estimate-1} because
$c^2(\Phi+\overline{\Phi})^2/(|c\Phi-i|^4+|c\Phi+i|^4)^2$ is
bounded for small $c$.

To establish \eqref{eq:hard-estimate-2}  observe that
\[
2i\,{\rm Im}\{(1+c^2{\Phi}^2)^2\} =
(1+c^2{\Phi}^2)^2-(1+c^2\overline{{\Phi}}^2)^2=
c^4({\Phi}^4-\overline{{\Phi}}^4)+2c^2({\Phi}^2-\overline{{\Phi}^2}) \, ,
\]
from which \eqref{eq:hard-estimate-2} follows directly. Finally,
\eqref{eq:hard-estimate-3} is a consequence of
\eqref{eq:hard-estimate-1} and \eqref{eq:hard-estimate-2}  because
for $\zeta=x+iy$ small then $|1-\zeta|
\leq |1-x|+2y^2$.

\end{proof}

\bibliography{hol-lift}

\end{document}